\documentclass[12pt,reqno]{amsart}

\usepackage{amssymb}
\usepackage[all]{xy}

\usepackage{hyperref}

\numberwithin{equation}{section}

\newtheorem{theorem}{Theorem}[section]

\newtheorem{corollary}[theorem]{Corollary}

\theoremstyle{definition}
\newtheorem{definition}[theorem]{Definition}

\begin{document}

\baselineskip=15pt

\title[Lie algebroid connections over rationally connected varieties]{Holomorphic
Lie algebroid connections over rationally connected varieties}

\author[I. Biswas]{Indranil Biswas}

\address{Department of Mathematics, Shiv Nadar University, NH91, Tehsil Dadri,
Greater Noida, Uttar Pradesh 201314, India}

\email{indranil.biswas@snu.edu.in, indranil29@gmail.com}

\author[A. Singh]{Anoop Singh}

\address{Department of Mathematical Sciences, Indian Institute of Technology (BHU), Varanasi 
221005, India}

\email{anoopsingh.mat@iitbhu.ac.in}



\date{}

\begin{abstract}
Take a holomorphic Lie algebroid $(V,\, \phi)$ over a rationally connected smooth complex
projective variety $X$. We show that --- under certain conditions --- a vector bundle $E$
over $X$ admits a $(V,\, \phi)$–connection if and only if $E$ is trivial. Moreover, we also
prove that under the same conditions, any $(V, \phi)$–connection over $X$ is flat.
\end{abstract}

\maketitle

\section{Introduction}

Let $X$ be a smooth projective variety over $\mathbb{C}$. Holomorphic 
connections on holomorphic vector bundles over $X$ were introduced by
Atiyah \cite{At2}. Holomorphic Lie algebroid connections (see Definition \ref{Def-3}) over $X$
provide a natural generalization of holomorphic connections, where
the tangent Lie algebroid $(TX, \, \mathrm{Id}_{TX})$ is replaced by
an arbitrary holomorphic Lie algebroid $(V,\, \phi)$ over $X$ (see Definition \ref{Def-1}).

It may be mentioned that by choosing the Lie algebroid appropriately, a wide range of other 
algebraic and differential geometric objects can be interpreted as Lie algebroid 
connections. For instance, Higgs bundles, \cite{Hi, Si1}, twisted Higgs bundles, 
\cite{Ni1,GGPN}, logarithmic connections, \cite{De,Ni2}, meromorphic connections, 
\cite{Bo,BS2}, and a broad subclass of Simpson's notion of $\Lambda$-modules, \cite{Si2, 
To2}, can all be understood as Lie algebroid connections for suitable choices of the Lie 
algebroid $(V,\, \phi)$.

As far as holomorphic connections are concerned, not every holomorphic vector bundle over a 
smooth complex projective variety $X$ admits a holomorphic connection. The existence of 
holomorphic connections is a highly restrictive condition. Moreover, even if a vector bundle 
over $X$ admits a holomorphic connection, it need not be flat. In contrast, if a holomorphic 
vector bundle over a compact Riemann surface admits a holomorphic connection, then the 
connection is necessarily flat. However, in general, a holomorphic vector bundle over a 
compact Riemann surface does not always admit a holomorphic connection. There is a criterion 
that determines when such a connection exists. To explain this criterion, for any holomorphic 
vector bundle $F$ on a compact connected Riemann surface $Y$, consider a holomorphic 
decomposition of $F$
$$
F \ =\ \bigoplus_{i=1}^\ell F_i
$$
into a direct sum of indecomposable holomorphic vector bundles. A theorem of Atiyah says that 
for any other holomorphic decomposition of $F$ into a direct sum of indecomposable 
holomorphic vector bundles, the isomorphism classes of the direct summands of it are simply a 
permutation of the isomorphism classes of $F_i$, $1\, \leq\, i\, \leq\, \ell$ \cite{At1}. The 
holomorphic vector bundle $F$ admits a holomorphic connection if and only if the degree of 
every $F_i$, $1\, \leq\, i\, \leq\, \ell$, is zero \cite{At2}, \cite{We}.

Further, if $X$ is a compact K\"ahler Calabi-Yau manifold, then any holomorphic vector bundle 
over $X$ admitting a holomorphic connection also admits a flat holomorphic connection 
\cite{BD}.

Here we consider holomorphic vector bundles over a rationally connected smooth complex 
projective variety and investigate Lie algebroid connections on them.
Recall that a complex projective variety $X$ is said to be 
{\it rationally connected} if any two points of $X$ can be joined by an irreducible rational 
curve on $X$; see \cite[Theorem 2.1]{KMM} for equivalent conditions.

Suppose that the anchor map $\phi\, :\, V\, \longrightarrow\, TX$ is surjective. Consider the following short exact sequence of holomorphic vector bundles
\begin{equation}\label{e.0}
0\, \longrightarrow\, S\,:=\, \text{Ker} (\phi) \, \longrightarrow\, V \, \stackrel{\phi}{\longrightarrow}\,
TX \, \longrightarrow\, 0.
\end{equation}
Then, we show the following (see Theorem \ref{thm:A}):

\begin{theorem}\label{thm:A.0}
Take a rationally connected smooth complex projective variety $X$
and a Lie algebroid $(V,\, \phi)$ over $X$ such that $\phi$ is surjective. Suppose that $S \,:=\,
\mathrm{Ker} (\phi)$ is strictly nef. A holomorphic vector bundle $E$ over $X$ admits a $(V,\,
\phi)$--connection if and only if $E$ is holomorphically trivial.
\end{theorem}

Under the assumptions of Theorem \ref{thm:A.0}, we prove the following (see Theorem \ref{thm:B}):

\begin{theorem}\label{thm:0.B}
Let $X$ be as in Theorem \ref{thm:A.0}, and let $(V, \phi)$ a Lie algebroid over $X$ such that
$\phi$ is surjective. Suppose that $S \,:=\, \text{Ker} (\phi)$ is strictly nef. Then, any
$(V,\, \phi)$-connection on a holomorphic vector bundle over $X$ is flat. (The vector
bundle is holomorphically trivial by Theorem \ref{thm:A.0}.)
\end{theorem}

\section{Lie algebroids and connections}

Let $X$ be a smooth projective variety over $\mathbb{C}$. The holomorphic tangent and
cotangent bundles of $X$ will be denoted by $TX$ and $\Omega^1_X$ respectively.

\begin{definition}\label{Def-1}
A Lie algebroid on $X$ is a holomorphic vector bundle $V\, \longrightarrow\, X$,
together with an ${\mathcal{O}}_X$--linear homomorphism $$\phi\, :\, V\, \longrightarrow\, TX,$$ called the {\it anchor
map}, and a structure of a $\mathbb{C}$--Lie algebra on the sheaf of locally defined holomorphic sections of $V$
$$
[-, \,-] \ :\ V \otimes_{\mathbb{C}} V \ \longrightarrow\ V
$$
such that $$[s,\, f\cdot t]\,=\, f \cdot [s,\, t]+\phi(s)(f) \cdot t$$ for all locally defined
holomorphic sections $s,\, t$ of $V$ and all locally defined holomorphic functions $f$ on $X$.
\end{definition}

Note that the pair $(TX,\, {\rm Id}_{TX})$ is a Lie algebroid; the Lie algebra structure on
$TX$ is given by the Lie bracket operation of vector fields.

\begin{definition}
\label{Def-2}
A Lie algebroid $(V,\, \phi)$ on $X$ is called \textit{split} if there is a holomorphic homomorphism
$\eta\, :\, TX\, \longrightarrow\, V$ such that
\begin{equation}\label{eta}
\phi\circ\eta\ =\ \text{Id}_{TX}.
\end{equation}
A Lie algebroid $(V,\, \phi)$ on $X$ is called \textit{nonsplit} if it is not split.
\end{definition}

Note that in the above Definition \ref{Def-2}, the homomorphism $\eta : TX \longrightarrow V$ 
 is not required to be a morphism of sheaves of Lie algebras.

\begin{definition}
\label{Def-3} \cite{ELW, Ma}
Take a Lie algebroid $(V,\, \phi)$ on $X$.
Let $\phi^*\,:\, \Omega^1_X \, \longrightarrow\, V^*$ be the dual of $\phi$.
A \emph{Lie algebroid connection} on a holomorphic vector bundle $E$ on $X$ 
is a $\mathbb{C}$--linear holomorphic map 
$$
\nabla \ :\ E\ \longrightarrow\ E\otimes V^*
$$
such that $$\nabla (fs) \,=\, f \nabla (s) + s\otimes \phi^*(df), $$ where $s$, as before, is any locally
defined holomorphic section of $E$ and $f$ is any locally defined
holomorphic function on $X$. A Lie algebroid connection on $E$ will also be called a $(V,\, \phi)$--connection on $E$ or simply
a $\phi$--connection on $E$.
\end{definition}

When $(V,\, \phi)\,=\, (TX,\, {\rm Id}_{TX})$, a $(V,\, \phi)$--connection on $E$ is a 
holomorphic connection on $E$ in the usual sense \cite{At2}.

For any Lie algebroid $(V,\,\phi)$, and any $n\, \geq\, 0$, there is a $\mathbb C$--linear homomorphism
$$
d^n_V\ :\ \bigwedge\nolimits^n V^* \ \longrightarrow\ \bigwedge\nolimits^{n+1} V^*
$$
which is constructed as follows:
\begin{equation}\label{d0}
d_V^0(f)\ =\ \phi^*(df)
\end{equation}
for any locally defined holomorphic function $f$
on $X$. To construct $d^n_V$ for $n\, \geq\, 1$, take locally defined holomorphic sections $\omega \,\in\,
\bigwedge^n V^*$ and $v_1,\,\cdots,\, v_{n+1} \,\in\, V$; then
\begin{equation}\label{d10}
d^n_V(\omega)(v_1,\, \cdots,\,v_{n+1})\ =\ \sum_{i=1}^n (-1)^{i+1}\phi(v_i)(\omega(v_1,\,\cdots,\,
\widehat{v}_i,\, \cdots,\, v_{n+1}))
\end{equation}
$$
+\sum_{1\le i<j\le n+1}(-1)^{i+j} \omega\left ([v_i,\, v_j],\, v_1,\,\cdots,\,
\widehat{v}_i,\, \cdots,\,\widehat{v}_j,\, \cdots,\, v_{n+1}\right).
$$
It can be easily checked that $d^{n+1}_V \circ d^{n}_V = 0$ for every $n \geq 0$, and it follows that 
$$\left (\bigwedge\nolimits^\bullet V^*,\,d_V \right) \ =\
\bigoplus_{n\geq 0} \left(\bigwedge\nolimits^n V^*,\,d^n_V\right)$$
is a differential graded complex; it
is called the Chevalley-Eilenberg-de Rham complex for $(V,\,\phi)$ (see \cite{BMRT}, \cite{LSX},
\cite{BR} for details).
Note that when $(V,\,\phi)\,=\, (TX,\,{\rm Id}_{TX})$, then $(\bigwedge\nolimits^\bullet V^*,\,d_V)$
is the holomorphic de Rham complex of $X$.

Given a Lie algebroid connection $\nabla \,:\, E\,\longrightarrow\, E\otimes V^*$, consider the
following composition of operators
$$
E\,\,\stackrel{\nabla}{\longrightarrow}\,\, E\otimes V^* \,\,\xrightarrow{\,\,\,\nabla \wedge {\rm Id}_{V^*}+
{\rm Id}_E\otimes d^1_V\,\,\,} \,\, E\otimes \bigwedge\nolimits^2 V^*.
$$
It is straightforward to check that this operator $E\, \longrightarrow\, E\otimes \bigwedge\nolimits^2 V^*$
is ${\mathcal O}_X$--linear, and hence it is given by a holomorphic section
$$
{\mathcal K}(\nabla)\,\, \in\,\, H^0(X,\, \text{End}(E)\otimes \bigwedge\nolimits^2 V^*).
$$
This section ${\mathcal K}(\nabla)$ is called the \textit{curvature} of $\nabla$. The 
Lie algebroid connection $\nabla$ is called \textit{integrable} (or \textit{flat}) if
we have ${\mathcal K}(\nabla)\,=\, 0$.

Note that on a smooth projective variety $X$ over $\mathbb{C}$, a $(V, \,
\phi)$-connection, where $\mathrm{rank} (V) \geq 2$, need not be flat even when the Lie 
algebroid $(V,\, \phi)$ is split.

We establish a $V$-jet bundle exact sequence. Consider the first order jet bundle $J^1(E)$ on $X$ whose fiber 
 over any $x\, \in\, X$ is the space of
all holomorphic sections of $E$ over the first order infinitesimal neighborhood
of $x$. Then, $J^1(E)$ fits into the following short exact sequence 
\begin{equation}\label{e0}
0\, \longrightarrow\, E\otimes \Omega^1_X \, \stackrel{\alpha}{\longrightarrow}\, J^1(E)
\,\stackrel{\beta}{\longrightarrow}\, J^0(E)\,=\,E \, \longrightarrow\, 0,
\end{equation}
where the homomorphism $\beta$ in \eqref{e0} is the restriction map that sends
a section of $E$ over the first order infinitesimal neighborhood of $x$ to the
evaluation of the section at $x$. Consider the homomorphism 
\begin{equation}\label{ed}
\Delta\ :\ E\otimes \Omega^1_X\ \longrightarrow\ (E\otimes V^*)\oplus J^1(E)
\end{equation}
that sends any $e\otimes w\, \in\, (E\otimes \Omega^1_X)_x$, where $x\, \in\, X$, $e\, \in\, E_x$ and $w\,\in\,
(\Omega^1_X)_x$, to
$$
(-e\otimes\phi^*_x(w),\,\, \alpha_x(e\otimes w))\,\,\in\,\, (E\otimes V^*)_x\oplus J^1(E)_x,
$$
where $\phi^*$ is the dual of the anchor map $\phi$ and $\alpha$ is the homomorphism in \eqref{e0}. Note
that $\Delta$ in \eqref{ed} is fiberwise injective because $\alpha$ is so. Now define the quotient
\begin{equation}\label{e10}
J^1_V(E) \,\, :=\,\, ((E\otimes V^*)\oplus J^1(E))/\Delta(E\otimes \Omega^1_X).
\end{equation}
Consider the composition of maps
$$
(E\otimes V^*)\oplus J^1(E)\, \longrightarrow\, J^1(E) \, \stackrel{\beta}{\longrightarrow}\, E,
$$
where $\beta$ is the homomorphism in \eqref{e0} and the first map is the natural projection. This composition
of homomorphisms evidently vanishes on $\Delta(E\otimes \Omega^1_X)\, \subset\, (E\otimes V^*)\oplus J^1(E)$ (see
\eqref{ed} for $\Delta$) because $\beta\circ\alpha\,=\, 0$ (see \eqref{e0}), and hence it produces a
surjective homomorphism
\begin{equation}\label{eb}
\beta'\,\,:\,\, J^1_V(E)\,\,=\,\,((E\otimes V^*)\oplus J^1(E))/\Delta(E\otimes \Omega^1_X)
\, \,\longrightarrow\,\, E.
\end{equation}
Let
\begin{equation}\label{ea}
\alpha'\,\,:\,\, E\otimes V^*\,\, \longrightarrow\,\, J^1_V(E)
\end{equation}
be the composition of maps
$$
E\otimes V^*\, \hookrightarrow\, (E\otimes V^*)\oplus J^1(E)\, \longrightarrow\, ((E\otimes V^*)\oplus 
J^1(E))/\Delta(E\otimes \Omega^1_X)\,=\, J^1_V(E),
$$
where $(E\otimes V^*)\oplus J^1(E)\, \longrightarrow\, J^1_V(E)$ is the natural
quotient map (see \eqref{e10}). Consider the composition of maps
\begin{equation}\label{ea2}
\beta'\circ\alpha'\, \, :\,\, E\otimes V^*\,\, \longrightarrow\,\, E,
\end{equation}
where $\alpha'$ and $\beta'$ are constructed in \eqref{ea} and \eqref{eb} respectively.
It is straightforward to check that $\beta'\circ\alpha'$ vanishes identically.
Note that the image of $\alpha'$ in
\eqref{ea} is precisely the kernel of $\beta'$ in \eqref{eb}. Consequently, we have
a short exact sequence
\begin{equation}\label{e-1}
0\, \longrightarrow\, E\otimes V^* \, \stackrel{\alpha'}{\longrightarrow}\, 
J^1_V(E) \, \stackrel{\beta'}{\longrightarrow}\, E \, \longrightarrow\, 0
\end{equation}
of holomorphic vector bundles on $X$. 

An alternative description of $J^1_V(E)$ is as follows.  Consider the following direct sum of $\mathbb{C}$-modules:
\begin{equation}\label{eq:jet}
J^1_{V}(E) :=  E\oplus(E\otimes V^*).
 \end{equation}
We can equip $J^1_{V}(E)$ with two $\mathcal{O}_X$-module structures. The first one is coordinate-wise
multiplication $$ f \cdot (s,  \sigma)  :=  (f s,   f \sigma),$$
and the second one is given by
\begin{equation}
\label{eq:mult}
f \cdot (s,  \sigma) := (f s,  f \sigma + s \otimes d_{V}f),
\end{equation}
where $s$ (respectively, $\sigma$) is a locally defined holomorphic section of $E$
(respectively, $E\otimes V^*$) and $f$ is a locally defined holomorphic function on $X$. The $\mathbb{C}$-module
$J^1_{V}(E)$ equipped with the $\mathcal{O}_X$-module structure in \eqref{eq:mult} is called the {\it first order $V$-jet bundle}
associated to $E$. The $V$-jet bundle $J^1_V(E)$ naturally fits into the short exact sequence \eqref{e-1}.

A holomorphic $(V,\, \phi)$--connection on $E$ is a holomorphic homomorphism
$\nabla\,:\, E\, \longrightarrow\, J^1_V(E)$ such that $\beta'\circ\nabla\,=\, {\rm Id}_E$.
Therefore, $E$ admits a holomorphic $(V,\, \phi)$--connection if and only if the short exact
sequence in \eqref{e-1} splits holomorphically. See \cite{ABKS} for a criterion for the existence
of Lie algebroid connections on holomorphic vector bundles over $X$.

\section{A criterion for flat $(V, \phi)$-connection}\label{LC}

In this section, we assume that $X$ is a rationally connected smooth
projective variety over $\mathbb{C}$.
Recall that a complex projective variety $X$ is said to be {\it rationally connected} if any 
two points of $X$ can be joined by an irreducible rational curve on $X$; see \cite[Theorem 2.1]{KMM}.

\begin{theorem}\label{thm:A}
Let $X$ be a rationally connected smooth complex projective variety, and let $(V,\, \phi)$ 
be a Lie algebroid over $X$ such that $\phi$ is surjective. Suppose that $S \,:=\, \text{Ker} 
(\phi)$ is strictly nef. A holomorphic vector bundle $E$ over $X$ admits a $(V,\,
\phi)$--connection if and only if $E$ is holomorphically trivial.
\end{theorem}

\begin{proof}
Suppose that $E$ admits a $(V,\, \phi)$--connection $\nabla$.
Consider a nonconstant morphism $ f \,:\, \mathbb{P}^1 \,\longrightarrow \,X$.
Then, we have the
differential $$df\,:\, T\mathbb{ P}^1\, \longrightarrow \, f^*TX.$$ The Lie algebroid
$(V,\, \phi)$ over $X$ gives a Lie algebroid over $\mathbb{ P}^1 $ as follows: Consider
the following commutative diagram
\begin{equation}\label{eq:cd1}
\begin{gathered}
\xymatrix{
V_f \ar[d]^{\widehat{\phi}} \ar[r]^{p} & f^*V \ar[d]^{f^*\phi} \\
T \mathbb{ P}^1 \ar[r]^{df} & f^*TX
}
\end{gathered}
\end{equation} 
where $V_f \, = \, T \mathbb{ P}^1\times_{f^*TX} f^*V$ is the fiber product, while
$$\widehat{\phi}\, :\, V_f \, \longrightarrow
\, T \mathbb{ P}^1$$ and $$p\,:\, V_f\, \longrightarrow \, f^*V$$
are the natural projections from the fiber product. Note that $V_f$ is the subbundle of 
$T\mathbb{ P}^1 \oplus f^*V$ whose fiber over any $y\,\in\, \mathbb{ P}^1 $ is the subspace 
of $T_y\mathbb{ P}^1 \oplus (f^*V)_y$ consisting of all $(a,\, b)$, where $a\,\in\, 
T_y\mathbb{ P}^1 $ and $b\, \in\, (f^*V)_y$, such that $$f^*\phi (b)\,=\, df (a).$$ The 
$\mathbb C$--Lie algebra structure on $V$ pulls back to a $\mathbb C$--Lie algebra on $f^*V$; 
the $\mathbb C$--Lie algebra structures on $T \mathbb{ P}^1 $ and $f^*V$ together produce a 
$\mathbb C$--Lie algebra structure on $V_f$. Thus $(V_f,\, \widehat{\phi})$ is a Lie 
algebroid over $\mathbb{ P}^1 $. Note that the maps $p$ and $\widehat{\phi}$ in \eqref{eq:cd1}
are Lie algebra 
structure preserving. Now, the pullback $f^*E\, \longrightarrow\, \mathbb{ P}^1 $ is equipped 
with a $(V_f,\, \widehat{\phi})$--connection $\widehat{\nabla}$ given by the following 
composition of maps
\begin{equation}\label{ewn}
f^*E \, \xrightarrow{\,\,\,\, f^*\nabla \,\,\,}\, f^*(E \otimes V^*) \, = \, (f^*E)\otimes (f^*V^*)\,
\xrightarrow{\,\,\, {\rm Id}_{f^*E}\otimes p^*\,\,\,}\, (f^*E)\otimes V_f^*,
\end{equation}
where $p^*$ is dual of the projection $p$ in \eqref{eq:cd1}. 

Note that the Lie algebroid $(V_f,\, \widehat{\phi})$ fits into the following short exact sequence 
\begin{equation}
\label{eq:v-f}
0\, \longrightarrow\, f^*S \, \longrightarrow \, V_f
\,\stackrel{\widehat{\phi}}{\longrightarrow}\, T \mathbb{P}^1 \, \longrightarrow\, 0,
\end{equation}
where $S \,=\, Ker(\phi)$; note that the surjectivity of $\widehat{\phi}$ follows from
the given condition that $\phi$ is surjective. Since $S$ is strictly nef over $X$,
and $df$ is not identically zero (recall that $f$ is a nonconstant map), it follows
that $f^*S$ is also strictly nef over $\mathbb{P}^1$ \cite[p.~84, Theorem 2.4]{Ha},
\cite[p.~360, Proposition 2.3]{Fu}.

The short exact sequence in \eqref{eq:v-f} splits holomorphically if and only if the extension
class in $\mathrm{H}^1(\mathbb{P}^1 ,\, K_{\mathbb{P}^1} \otimes f^*S)$ vanishes.
From Serre duality, we have 
$$
\mathrm{H}^1(\mathbb{P}^1 ,\, K_{\mathbb{P}^1} \otimes f^*S) \, = \, \mathrm{H}^0(\mathbb{P}^1,\, (f^*S)^* )^*,
$$
and $\mathrm{H}^0(\mathbb{P}^1,\, (f^*S)^* ) = 0$, because $f^*S$ is strictly nef over 
$\mathbb{P}^1$ \cite[Theorem 2.10]{LOY}, where $ (f^*S)^* $ dual of $ f^*S$. Therefore 
$\mathrm{H}^1(\mathbb{P}^1 ,\, K_{\mathbb{P}^1} \otimes f^*S ) = 0$, and hence the short 
exact sequence \eqref{eq:v-f} splits holomorphically. Fix a holomorphic homomorphism
$\gamma \, :\, T \mathbb{P}^1 \, \longrightarrow \, 
V_f $ of vector bundles such that $\widehat{\phi} \circ \gamma \, = \, 
\mathrm{Id}_{ T \mathbb{P}^1}$, where $\widehat{\phi}$ is the homomorphism
in \eqref{eq:v-f}. Consider the composition of maps
$$
f^*E \ \xrightarrow{\,\,\,\, \widehat{\nabla} \,\,\,}\ (f^*E)\otimes V_f^* \
\xrightarrow{\,\,\, {\rm Id}_{f^*E}\otimes \gamma^*\,\,\,}\ (f^*E)\otimes K_{\mathbb{P}^1}
$$
(recall that $\widehat{\nabla}$ is the composition of maps in \eqref{ewn})
which gives an usual holomorphic
connection on $f^*E\, \longrightarrow\, \mathbb{P}^1$. If a holomorphic vector
bundle $W$ on ${\mathbb P}^1$ admits a holomorphic connection, then $W$ is holomorphically
trivial. This is because any holomorphic connection on a Riemann surface is flat, and
${\mathbb P}^1$ is simply connected, so $W$ is holomorphically trivial.
Therefore, $f^*E$ is holomorphically trivial. Since the morphism $f$ is 
an arbitrary nonconstant map, it follows that $E$ is trivial \cite[Proposition 1.2]{BS1}.
\end{proof}

\begin{theorem}\label{thm:B}
Let $X$ be a rationally connected smooth complex projective variety, and let $(V,\, \phi)$ 
be a Lie algebroid over $X$ such that $\phi$ is surjective. Suppose that $S \,:=\, \text{Ker} 
(\phi)$ is strictly nef. Then, any $(V,\, \phi)$--connection on $E$ over $X$ is flat.
\end{theorem}

\begin{proof}
Since $X$ is rationally connected, there are rational curves $f\, :\, {\mathbb P}^1\,
\longrightarrow\, X$ such that the pullback $f^*TX\, \longrightarrow\,{\mathbb P}^1$ is ample
\cite[p.~433--434, Theorem 2.1]{KMM} (see (5) of \cite[Theorem 2.1]{KMM} on p.~434). Any
ample vector bundle $W$ on ${\mathbb P}^1$ is globally generated, because $W$
is a direct sum of holomorphic line bundles \cite[p.~122, Th\'eor\`eme 1.1]{Gr}, and each
line bundle in the direct sum is of positive degree \cite[p.~84, Theorem 2.4]{Ha}.
Consequently, the
union of all rational curves $f\, :\, {\mathbb P}^1\,
\longrightarrow\, X$ such that the pullback $f^*TX\, \longrightarrow\,{\mathbb P}^1$
is ample is a Zariski dense subset of $X$ \cite{KMM} .

Take any rational curve $f\, :\, {\mathbb P}^1\,
\longrightarrow\, X$ such that $f^*TX\, \longrightarrow\,{\mathbb P}^1$ is ample.
We will show that the vector bundle $f^*V\, \longrightarrow\,{\mathbb P}^1$ is ample.

Pulling back \eqref{e.0} we have following exact sequence on ${\mathbb P}^1$:
\begin{equation}\label{es}
0\ \longrightarrow\ f^*S\ \longrightarrow\ f^*V \ \xrightarrow{\,\,\,f^*\phi\,\,\,}\
f^*TX \ \longrightarrow\ 0.
\end{equation}
Decompose the vector bundle $f^*S$ in \eqref{es} into a direct sum of line bundles
\cite[p.~122, Th\'eor\`eme 1.1]{Gr}. Each line
bundle in this direct sum is of positive degree because $S$ is strictly nef. Consequently,
$f^*S$ is ample. Since $f^*TX$ is also ample, we now conclude that $f^*V$ is ample
\cite[p.~84, Theorem 2.4]{Ha}.

Let $\nabla$ be a Lie algebroid connection on $E$. Let
\begin{equation}\label{es3}
{\mathcal K}(\nabla)\ \in\ H^0(X,\, \text{End}(E)\bigwedge\nolimits ^2 V^*)
\end{equation}
be the curvature of $\nabla$. The vector bundle $\bigwedge^2 f^*V^*$ is a direct sum of line
bundles of negative degrees because $f^*V$ is ample. Hence we have
\begin{equation}\label{es2}
H^0({\mathbb P}^1,\, \bigwedge\nolimits^2 f^*V^*)\ =\ 0. 
\end{equation}
{}From Theorem \ref{thm:A} we know that $f^*E$ is a trivial vector bundle.
Therefore, from \eqref{es2} it follows immediately that
\begin{equation}\label{es4}
{\mathcal K}(\nabla)\big\vert_{f({\mathbb P}^1)}\ = \ 0
\end{equation}
(see \eqref{es3}).
Since the union of all rational curves $f\, :\, {\mathbb P}^1\,
\longrightarrow\, X$ such that the pullback $f^*TX\, \longrightarrow\,{\mathbb P}^1$
is ample is a Zariski dense subset of $X$, from \eqref{es4} it follows that
${\mathcal K}(\nabla)\,=\, 0$. In other words, $\nabla$ is a flat Lie algebroid connection.
\end{proof}

Theorem \ref{thm:A} and Theorem \ref{thm:B} together yield the following:

\begin{corollary}\label{cor1}
Let $X$ be a rationally connected smooth complex projective variety. Let
$E$ be a holomorphic vector bundle on $X$ equipped with a usual holomorphic connection
$\nabla$. Then $E$ is holomorphically trivial and $\nabla$ is integrable (same as flat).
\end{corollary}

\begin{proof}
Theorem \ref{thm:A} implies that $E$ is holomorphically trivial, and Theorem \ref{thm:B}
implies that $E$ is flat.
\end{proof}

\section*{Acknowledgements}
The authors would like to thank the referee for carefully going through the manuscript and pointing out some errors in an earlier version. I.B. is partially supported by a J. C. Bose Fellowship (JBR/2023/000003). A.S. is partially supported by ANRF /ARGM/2025/000670/MTR.

\end{document}